\documentclass[a4paper,12pt]{amsart}
\usepackage{amsmath, amssymb, amsthm}
\usepackage{verbatim}
\usepackage{hyperref}
\usepackage{enumerate}
\usepackage{xypic}

\newcounter{theorem}
\newtheorem{thm}[theorem]{Theorem}
\newtheorem{lemma}[theorem]{Lemma}
\newtheorem{prop}[theorem]{Proposition}
\newtheorem{cor}[theorem]{Corollary}

\theoremstyle{remark}
\newtheorem*{remark*}{Remark}
\newtheorem{remark}[theorem]{Remark}

\numberwithin{equation}{section}
\numberwithin{theorem}{section}

\newcommand{\e}{\epsilon}

\newcommand{\R}{\mathbb{R}}
\newcommand{\C}{\mathbb{C}}
\newcommand{\N}{\mathbb{N}}
\renewcommand{\setminus}{\backslash}

\renewcommand{\emptyset}{\varnothing}

\newcommand{\id}{\mathrm{id}}

\newcommand{\ev}{\mathrm{ev}}

\newcommand{\dr}{\mathrm{dr}}

\newcommand{\labelledthing}[2]{\hspace{4pt}\buildrel {#2} \over #1 \hspace{3pt}} 

\newcommand{\labelledrightarrow}{\labelledthing{\longrightarrow}}

\begin{document}

\title{High-dimensional $\mathcal Z$-stable AH algebras}
\author{Aaron Tikuisis}
\address{Institute of Mathematics\\
University of Aberdeen\\
Aberdeen, UK AB24 3UE}
\urladdr{http://homepages.abdn.ac.uk/a.tikuisis/}
\email{a.tikuisis@abdn.ac.uk}

\begin{abstract}
It is shown that a $C^*$-algebra of the form $C(X,U)$, where $U$ is a UHF algebra, is not an inductive limit of subhomogeneous $C^*$-algebras of topological dimension less than that of $X$.
This is in sharp contrast to dimension-reduction phenomenon in (i) simple inductive limits of such algebras, where classification implies low-dimensional approximations, and (ii) when dimension is measured using decomposition rank, as the author and Winter proved that $\dr(C(X,U)) \leq 2$.
\end{abstract}

\thanks{The author is supported by an NSERC PDF}



\maketitle

\section{Introduction}

Consider $C^*$-algebras that take the form of a direct sum of algebras of continuous functions from a topological space to a matrix algebra; call this class $\mathcal C$.
Now consider the class A$\mathcal C$ of algebras that are inductive limits of algebras in $\mathcal C$.
Such algebras, including AF, AI, A$\mathbb T$, and (some) AH algebras, have arisen naturally, for instance, as crossed products of the Cantor set or the circle by minimal homeomorphisms.
However, the present purpose of considering this class of $C^*$-algebras is as a test case for phenomena among the broader, less-understood class of finite nuclear $C^*$-algebras.

A mixture of classification and other arguments has shown that there is a dichotomy amongst the simple $C^*$-algebras in A$\mathcal C$, dividing them into algebras of low and high topological dimension.
Classification arguments, on the one hand, show that for a simple algebra in A$\mathcal C$, if it is an inductive limit of building blocks (in $\mathcal C$) with bounded topological dimension (or even ``slow dimension growth''), or if it is $\mathcal Z$-stable, then it is an inductive limit of algebras in $\mathcal C$ with topological dimension at most three \cite{ElliottGongLi:AHclassification,Gong:SimpleReduction,Lin:LocalAH}.
(Note that one can show, without classification, that slow dimension growth implies $\mathcal Z$-stability; see \cite{Toms:rigidity,Toms:stability,Winter:pure}, so $\mathcal Z$-stability should be viewed as the weakest of these hypotheses.)
By a general $\mathcal Z$-stability theorem of Winter \cite{Winter:pure}, it follows that this is also the case for simple algebras in A$\mathcal C$ with finite nuclear dimension (or decomposition rank).
Villadsen first produced examples of simple algebras in A$\mathcal C$ that don't have bounded-dimension building blocks \cite{Villadsen:Perforation}; these examples, and their high-dimensional properties, were further analyzed by Toms and Winter \cite{TomsWinter:V1}.

The aforementioned dimension reductions (going from bounded dimension/slow dimension/$\mathcal Z$-stability to dimension at most three) owes itself to (i) simplicity and (ii) dimension-reducing effects of tensoring with $\mathcal Z$.
Villadsen's high-dimensional algebras show that simplicity alone does not produce dimension reduction.
Moreover, the author and Winter showed that dimension-reduction, in terms of decomposition rank (and therefore also nuclear dimension), does occur even for nonsimple, $\mathcal Z$-stable algebras in A$\mathcal C$: if $A \in \text{A}\mathcal C$ then $\dr(A \otimes \mathcal Z) \leq 2$ \cite{TW:Zdr}.
This result has recently been extended, by the author and others, to algebras $A$ that are inductive limits of subhomogeneous algebras \cite{ENSTW}.

The main result here is that certain non-simple $\mathcal Z$-stable algebras in A$\mathcal C$ (namely, algebras of the form $C(X,U)$ where $U$ is a UHF algebra) cannot be approximated by low-dimensional algebras in $\mathcal C$--or even from the significantly larger class of subhomogeneous $C^*$-algebras.
This result clarifies the role played by simplicity in classification results such as \cite{ElliottGongLi:AHclassification,Gong:SimpleReduction,Lin:LocalAH}.
It also provides the first example of a $\mathcal Z$-stable approximately subhomogeneous algebra $A$ which cannot be approximated by subhomogeneous algebras with bounded decomposition rank; it is expected (and largely entailed by classification conjectures) that this phenomenon cannot occur in the simple case.

In \cite{BBSTWW}, the author and others showed that many simple $\mathcal Z$-stable $C^*$-algebras have decomposition rank $1$, including a number of approximately subhomogeneous algebras that cannot be approximated by subhomogeneous algebras with decomposition rank less than $2$.
This and the present article give complementary accounts of the same phenomenon: that decomposition rank does not produce the same $K$-theoretic restrictions as subhomogeneous dimension.

In \cite{KirchbergRordam:pi3}, Kirchberg and R\o rdam devised a way to show that for any commutative $C^*$-algebra $C$, $C \otimes 1_{\mathcal O_2}$ can be approximated within $C \otimes \mathcal O_2$ by commutative $C^*$-algebras with one-dimensional spectrum.
They use this result to show that a large swathe of strongly purely infinite, non-simple algebras are approximated by algebras in $\mathcal C$ with one-dimensional spectrum; the result also plays a crucial role in the dimension-reduction result of \cite{TW:Zdr}.
Kirchberg and R\o rdam's result rests mainly on the fact that $\mathcal O_2$ has trivial $K$-theory, which is closely tied to the existence of an $\mathcal O_2$-relativized retract $D^2 \to S^1$; that is, a solution to
\[
\xymatrix{
C(S^{1}) \ar@{-->}[r]^-{\exists} \ar[rd]_-{\id \otimes 1_{A}} & C(D^2,\mathcal O_2) \ar[d]^{f \mapsto f|_{S^{1}}} \\
& C(S^{1},\mathcal O_2),
}\]
It is not difficult to adapt their argument (as we do in Section \ref{sec:4}) to show that if $A$ is a $C^*$-algebra for which we can solve
\begin{equation}
\label{eq:RelRetract}
\xymatrix{
C(S^{n-1}) \ar@{-->}[r]^-{\exists} \ar[rd]_-{\id \otimes 1_{A}} & C(D^n,A) \ar[d]^{f \mapsto f|_{S^{n-1}}} \\
& C(S^{n-1},A),
} \end{equation}
then for any $n$-dimensional space $X$, $C(X) \otimes 1_A$ can be approximated in $C(X,A)$ by commutative $C^*$-algebras with $(n-1)$-dimensional spectrum.
Our main result arises by showing that the converse is true: if $C(X) \otimes 1_A$ can be approximated in $C(X,A)$ by $(n-1)$-dimensional commutative (or even subhomogeneous) algebras then there is an $A$-relativized retract of $D^n$ onto $S^{n-1}$, as in \eqref{eq:RelRetract}.

Let us introduce our notation precisely before clearly stating the main result.
For a $C^*$-algebra $A$, a $C^*$-subalgebra $B$, and a class $\mathcal S$ of $C^*$-algebras, we say that $B$ is \textbf{locally approximated} in $A$ by $C^*$-algebras in $\mathcal S$ if the following holds: for every finite subset $\mathcal F$ of $B$ and every $\e > 0$, there exists a $C^*$-subalgebra $C$ of $A$ such that $C \in \mathcal S$ and $\mathcal F \subset_\e C$.
Note that if $A$ is an inductive limit of algebras in $\mathcal S$ (or an inductive limit of inductive limits of algebras in $\mathcal S$, etc.) then it is locally approximated by algebras in $\mathcal S$.

The \textbf{topological dimension} of a commutative $C^*$-algebra $C(X)$ means, naturally, the dimension of $X$.
More generally, if $A$ is subhomogeneous, then the value of its decomposition rank and nuclear dimension coincide, and also agrees with a value obtained from the primitive ideal space \cite{Winter:drSH} (see also \cite[Corollary 3.18]{ENSTW} for an alternate proof); we continue to use the term topological dimension to refer to this value.

In the sequel, we identify $C(X)$ with the subalgebra $C(X, \mathbb C1_A)$ of $C(X,A)$.

\begin{thm}
\label{thm:MainThm}
Let $A$ be a unital $C^*$-algebra and let $n\in\N$.
The following are equivalent.
\begin{enumerate}[(i)]
\item
\label{MainThm1}
For every $n$-dimensional compact Hausdorff space $X$, $C(X)$ is approximated in $C(X,A)$ by commutative $C^*$-algebras of topological dimension at most $n-1$;
\item
\label{MainThm2}
For every $n$-dimensional compact Hausdorff space $X$, $C(X)$ is approximated in $C(X,A)$ by subhomogeneous $C^*$-algebras of topological dimension at most $n-1$;
\item
\label{MainThm3}
There exists a $*$-homomorphism $\phi:C(S^{n-1}) \to C(D^n,A)$ such that the diagram
\[
\xymatrix{
C(S^{n-1}) \ar[r]^\phi \ar[dr]_{\iota} & C(D^n,A) \ar[d]_{r} \\
& C(S^{n-1},A),
}
\]
commutes, where $\iota$ denotes the inclusion map and $r:C(D^n,A) \to C(S^{n-1},A)$ denotes the restriction map.
\end{enumerate}
If $A \cong A \otimes A$ (using any tensor norm), then these are also equivalent to statement (i) with the words ``$n$-dimensional'' removed; that is, that for any Hausdorff space $X$, $C(X)$ is approximated in $C(X,A)$ by commutative $C^*$-algebras of topological dimension at most $n-1$.
\end{thm}

Well-known topological arguments can be used to reformulate (\ref{MainThm3}) as: the inclusion map $C_0(\R^{n-1}) \to C_0(\R^{n-1},A)$ is nullhomotopic.
By an easy $K$-theoretic obstruction to (\ref{MainThm3}), we obtain:

\begin{cor}
\label{cor:MainCor}
Neither $C(D^n,U)$ (where $U$ is UHF) nor $C(D^n,\mathcal Z)$ are locally approximated by subhomogeneous algebras of topological dimension less than $n$.
\end{cor}

The remainder of the article is devoted to proving Theorem \ref{thm:MainThm}.
In Section \ref{sec:2}, we derive from (\ref{MainThm2}) a covering-dimension-related condition, that $C(X,A)$ contains arbitrarily fine, \textit{commuting} $n$-colourable partitions of unity, in the sense of \cite[Proposition 3.2 (iv)]{TW:Zdr} (this in particular implies (\ref{MainThm1})).
A significant part of this argument is handling the subhomogeneous case, which is aided by \cite[Theorem 2.15]{ENSTW}, which allows us to work with quotients of noncommutative cell complexes in place of arbitrary subhomogeneous approximants.
In Section \ref{sec:3}, we use the commuting partitions of unity from Section \ref{sec:2} to obtain the $A$-relativized retract of $D^n$ onto $S^{n-1}$ (i.e.\ condition (\ref{MainThm3}); the main part of this argument is showing that (\ref{MainThm3}) holds approximately.
Finally, Section \ref{sec:4} shows that (\ref{MainThm3}) implies (\ref{MainThm1}), which is essentially a reproduction of Kirchberg and R\o rdam's argument in the proof of \cite[Proposition 3.5]{KirchbergRordam:pi3}.

For elements $a$ and $b$ of a $C^*$-algebra $A$, and a real number $\e > 0$, we will use the notation $a \approx_\e b$ to mean $\|a-b\|<\e$.
If $A$ is a $C^*$-algebra and $a \in A$, then $\mathrm{sp}(a)$ denotes the spectrum of $a$.
For a compact topological space $X$, $CX$ denotes the cone 
\begin{equation}
\label{eq:Cone} CX:= X \times [0,1]/(X \times \{1\}).
\end{equation}

\subsection{Acknowledgements}
The author would like to thank Luis Santiago for discussions that prompted this research, and Ilijas Farah for helpful feedback on an earlier preprint.

\section{Simplicial complexes}
\label{sec:SimpComplex}

In this section we review some basic notions and results about simplicial complexes for use in the sequel.
For a more thorough treatment, see \cite[VIII.5]{Dugundji}.
Let $V$ be a finite set.
Recall that an (abstract, finite) simplicial complex on the vertex set $V$ is a subset $N$ of the power set of $V$, containing all singletons, and closed under taking subsets.
If $N$ is a simplicial complex, each set in $N$ is called a face or a simplex of $N$.
If $N$ is a simplicial complex with vertex set $V$, we use $|N|$ to denote its geometric realization.
Concretely, this can be realized as the subset of $[0,1]^V$ consisting of all $(\lambda_v)_{v\in V}$ such that:
\begin{enumerate}
\item $\sum_v \lambda_v = 1$; and 
\item For any $v_1,\dots,v_k \in V$, if $\{v_1,\dots,v_k\} \not\in N$ then
\[ \lambda_{v_1} \cdots \lambda_{v_k} = 0. \]
\end{enumerate}
For a vertex $v_0 \in V$, $\mathrm{Star}(v_0)$ denotes the star around $v_0$, that is, the open subset of $|N|$ made up of the union of the interiors of the realization of faces containing $v_0$.
With the above description of the geometric realization, $\mathrm{Star}(v)$ is precisely all $(\lambda_v)_{v\in V} \in |N|$ such that $\lambda_{v_0} \neq 0$.

We say that $N$ has dimension at most $n$ if every face of $N$ contains at most $n+1$ vertices.
This is the same as saying that the covering dimension of $|N|$ is at most $n$.

Let $N,N'$ be simplicial complexes with vertex sets $V,V'$ respectively.
A morphism from $N$ to $N'$ is a map $\alpha:V \to V'$ which takes each set in $N$ to a set in $N'$.
A morphism from $N$ to $N'$ induces a continuous map $\hat\alpha:|N|$ to $|N'|$ which takes the realization of a simplex $\sigma \in N$ to the realization of $\alpha(\sigma)$, namely
\[ \hat\alpha((\lambda_v)_{v\in V}) = (\mu_{v'})_{v' \in V'} \]
where $\mu_{v'} = \sum_{v \in \alpha^{-1}(v)} \lambda_v$.
In particular, $\hat\alpha(\mathrm{Star}(v))=\mathrm{Star}(\alpha(v))$ for every $v\in V$.

The following result follows almost immediately from the above description of the geometric realization of a simplicial complex.

\begin{prop}
\label{prop:SimpComplexCstar}
Let $N$ be a simplicial complex with vertex set $V$.
Then $C(|N|)$ can be canonically identified with the universal $C^*$-algebra with generator set $\{e_v \mid v \in V\}$ and the following relations:
\begin{enumerate}
\item Each $e_v$ is a positive contraction;
\item The $e_v$s commute;
\item For each $v_1,\dots,v_k \in V$, if $\{v_1,\dots,v_k\} \not\in N$ then
\[ e_{v_1}\dots e_{v_k} = 0; \quad \text{and} \]
\item $\sum_{v \in V} e_v = 1$.
\end{enumerate}
With this identification, for each vertex $v\in V$, we have
\[ \mathrm{Star}(v) = e_v{-1}((0,1]). \]
\end{prop}

Given a commuting finite partition of unity $(f_i)_{i \in I}$ in a $C^*$-algebra $A$ (that is, a finite set of commuting positive elements that sum to $1_A$), we define the \textbf{nerve} of $(f_i)_{i\in I}$ to be the following simplicial complex $N((f_i)_{i\in i})$ with vertex set $I$:
\[ N((f_i)_{i\in I}) := \Big\{\,\{i_1,\dots,i_k\} \subseteq I \mid f_{i_1}\dots f_{i_k} \neq 0\Big\}. \]
By Proposition \ref{prop:SimpComplexCstar}, we obtain a canonical map $C(|N((f_i)_{i\in I})|) \to C^*(\{f_i \mid i \in I\})$, sending $e_i$ to $f_i$.

\section{Subhomogeneous approximations and covering dimension}
\label{sec:2}

In this section, we prove that conditions (\ref{MainThm1}) and (\ref{MainThm2}) of Theorem \ref{thm:MainThm} are equivalent, and are equivalent to another, somewhat technical, covering-dimension-related condition.

We begin with a technical lemma, that will allow us to turn approximations in noncommutative cell complexes (defined in \cite{ENSTW}) into commutative approximations.

\begin{lemma}
\label{lem:CommGluing}
Let $X$ be a compact metric spaces, let $A_0$ be a unital $C^*$-algebra, and let $\phi:A_0 \to C(X,M_k)$ be a $*$-homomorphism for some $k\in\N$.
Define $A$ by the pullback diagram
\[
\begin{array}{rcl}
A & \labelledrightarrow{\rho} & C(CX,M_k) \\
{\scriptstyle \lambda} \downarrow && \downarrow {\scriptstyle f \mapsto f|_{X \times \{0\}}} \\
A_0 & \labelledrightarrow{\phi} & C(X,M_k).
\end{array}
\]
Let $f_1,\dots,f_m \in A_+$ and let $\e > 0$.
Suppose that $B_0$ is a commutative $C^*$-subalgebra of $A_0$, $1_{A_0} \in B_0$, and $g_0,\dots,g_m \in (B_0)_+$ are such that
\[ \mathrm{sp}(\pi(g_i)) \subset [\min \mathrm{sp}(\pi(\lambda(f_i))) - \e, \max \mathrm{sp}(\pi(\lambda(f_i)))+\e] \]
for all representations $\pi$ of $A_0$ and all $i=1,\dots,m$.
There exists a commutative unital $C^*$-subalgebra $B$ of $A$ and $h_1,\dots,h_m \in B$ such that $\lambda(h_i) = g_i$ and
\begin{equation}
\label{eq:CommGluingCrucialRTP}
\mathrm{sp}(\pi(h_i)) \subset [\min \mathrm{sp}(\pi(f_i)) - 2\e, \max \mathrm{sp}(\pi(f_i))+2\e]
\end{equation}
for all representations $\pi$ of $A$ and all $i=1,\dots,m$.
Moreover, the dimension of the spectrum of $B$ is at most the maximum of the dimension of $CX$ and the dimension of the spectrum of $B_0$.
\end{lemma}

\begin{proof}
Fix $\eta > 0$ such that
\[ \rho(f_i)(x,t) \approx_\e \rho(f_i)(x,s) \]
for all $s,t \in [0,\eta]$.
Define $\theta:C(X \times [0,\eta],B_0) \to C(X \times [0,\eta],M_k)$ by
\[ \theta(c)(x,t) = \phi(c(x,t))(x), \]
and set $D$ to be the image of $\theta$.
Evidently, $D$ is commutative.
Due to the way $\theta$ is defined, each irreducible representation of $D$ factors through a point evaluation at some unique point in $X \times [0,\eta]$, and this provides a finite-to-one continuous map from the spectrum of $D$ to $X \times [0,\eta]$.
It follows by \cite[Theorem 1.12.4]{Engelking:DimTheory} that the topological dimension of $D$ is at most the dimension of $X \times [0,\eta]$ (which is the same as the dimension of $CX$).
Set
\begin{align*}
B := &\{(c,d,b) \in C(CX|_{[\eta,1]}) \oplus D \oplus B_0 \\
&\quad \mid
c(x,\eta) = d(x,\eta)\text{ for all }x\in X\text{ and }
\phi(b)=d|_{X \times \{0\}}\}.
\end{align*}
$B$ is evidently commutative, and it embeds into $A$ in a pretty straightforward way: by sending $(c,d,b)$ to $(e,b) \in C(CX,M_k) \oplus A_0$ where
\[ e|_{CX|_{[\eta,1]}} = c \quad \text{and}\quad e|_{CX|_{[0,\eta]}} = d. \]

The spectrum of $B$ is given by gluing together $CX|_{[\eta,1]}$, the spectrum of $D$, and the spectrum of $B_0$.
Hence, the spectrum of $B$ has dimension at most
\begin{align*}
\max\{\dim(CX), \dim(\text{spectrum of }B_0)\}.
\end{align*}

For $i=1,\dots,m$, define $d_i \in C(X \times [0,\eta],B_0)$ by
\[ d_i(x,t) = \frac t\eta \|\rho(f_i)(x,t)\| 1_{B_0} + (1-\frac t\eta)g_i, \]
for all $(x,t) \in X \times [0,\eta]$.
Note that
\begin{align*}
\theta(d_i)(x,\eta) &= \|\rho(f_i)(x,\eta)\|,\quad\text{for all }x\in X, \\
\theta(d_i)(x,0) &= g_i,\quad\text{for all }x \in X,
\end{align*}
and that $\theta(d_i)(x,t)$ is a convex combination of $\|\rho(f_i)(x,t)\|1_{M_k}$ and $\phi(g_i)(x)$ for all $(x,t) \in X \times [0,\eta]$.
Define $c_i:CX|_{[\eta,1]} \to M_k$ by
\[ c_i(x) := \|\rho(f_i)(x)\| \]
Then evidently, $h_i:=(c_i,\theta(d_i),g_i) \in B$ and it is not hard to see that \eqref{eq:CommGluingCrucialRTP} holds.
\end{proof}

Let $X$ be a compact metric space and let $A$ be a $C^*$-algebra.
Let $\mathcal U$ be an open cover of $X$.
Following \cite[Proposition 3.2 (iv)]{TW:Zdr}, an \textbf{$(n+1)$-colourable partition of unity subordinate to $\mathcal U$} means positive elements $b_j^{(i)} \in C(X,A)$ for $i=0,\dots,n, j=1,\dots,r$, such that:
\begin{enumerate}
\item for each $i$, the elements $b_1^{(i)},\dots,b_r^{(i)}$ are pairwise orthogonal,
\item for each $i,j$, the support of $b_j^{(i)}$ is contained in some open set in the given cover $\mathcal{U}$, and
\item $\sum_{i,j} b_j^{(i)} = 1$.
\end{enumerate}

\begin{prop}
\label{prop:CoveringTFAE}
Let $X$ be a compact metric space, let $n \in \N$, and let $A$ be a unital $C^*$-algebra.
The following are equivalent:
\begin{enumerate}[(a)]
\item 
\label{CoveringTFAEa}
$C(X)$ is approximated in $C(X,A)$ by subhomogeneous algebras of topological dimension at most $n$;
\item
\label{CoveringTFAEb}
$C(X)$ is approximated in $C(X,A)$ by abelian $C^*$-algebras whose spectrum has dimension at most $n$;
\item
\label{CoveringTFAEc}
For every open cover $\mathcal U$ of $X$, there exists a \textbf{commuting} $(n+1)$-colourable partition of unity subordinate to $\mathcal U$.
\end{enumerate}
\end{prop}

\begin{remark}
In \cite[Proposition 3.2]{TW:Zdr}, it is shown that the existence of $(n+1)$-colourable approximate partitions of unity in $C(X,A)$ is equivalent to $\mathrm{dr} (C(X) \subset C(X,A)) \leq n$.
Condition (\ref{CoveringTFAEc}) is a notable strengthening, in that the partition of unity is asked to be commuting.
(Note that the difference between approximate and exact commuting partitions of unity is moot, since any commuting approximate partition of unity can be turned into a commuting exact partition of unity by functional calculus.)
Corollary \ref{cor:MainCor} and \cite[Theorem 4.1]{TW:Zdr} shows that in many cases, 
(\ref{CoveringTFAEc}) is not equivalent to the weaker condition of non-commuting approximate partitions of unity.
\end{remark}

\begin{proof}
(\ref{CoveringTFAEa}) $\Rightarrow$ (\ref{CoveringTFAEb}):
Let $\mathcal F = \{a_1,\dots,a_m\} \subset C(X)$ be a finite set and let $\e > 0$.
Without loss of generality, $\mathcal F$ consists of positive elements.
By \cite[Theorem 2.15]{ENSTW}, every subhomogeneous algebra is approximated by images of noncommutative cell complexes (defined in \cite[Definition 2.1]{ENSTW}) of the same topological dimension.
Therefore, this result and (\ref{CoveringTFAEa}) provide the existence of a noncommutative cell complex $B$ of dimension at most $n-1$, a $*$-homomorphism $\phi:B \to C(X,A)$, and elements $b_1,\dots,b_m \in B$ be such that $a_i \approx_{\e/2} \phi(b_i)$ for $i=1,\dots,m$.
By an inductive argument with Lemma \ref{lem:CommGluing} as the inductive step, there exists a commutative $C^*$-subalgebra $C$ of $B$, whose spectrum has dimension at most $n-1$, and $c_1,\dots,c_m \in C_+$ such that, for every representation $\pi$ of $B$,
\begin{equation}
\label{eq:CoveringTFAEspContain}
\mathrm{sp}(\pi(c_i)) \subseteq [\min \mathrm{sp}(\pi(b_i))-\e/2, \max\mathrm{sp}(\pi(b_i))+\e/2].
 \end{equation}
For each $x\in X$, $\phi(b_i)(x) \approx_{\e/2} a_i(x) \in \mathbb C$, so that
\[  \mathrm{sp}(\phi(b_i)(x)) \subseteq (a_i(x)-\e/2, a_i(x)+\e/2). \]
Using this and \eqref{eq:CoveringTFAEspContain} with $\pi=\phi(\cdot)(x)$, this yields that
\[ \mathrm{sp}(\phi(c_i)(x)) \subseteq (a_i(x)-\e, a_i(x)+\e), \]
and thus, $\phi(c_i) \approx_\e a_i$.

Hence, $C$ and $c_1,\dots,c_m$ witness that (\ref{CoveringTFAEb}) holds.

(\ref{CoveringTFAEb}) $\Rightarrow$ (\ref{CoveringTFAEc}):
Let $\mathcal F$ be a finite partition of unity such that, for each $f \in \mathcal F$, there exists $U_f \in \mathcal{U}$ such that $\mathrm{supp}\, f \subset U_f$.
Use (\ref{CoveringTFAEb}) to obtain a space $Y$ of dimension at most $n$ and a $*$-homomorphism $\phi:C(Y) \to C(X,A)$ such that $\mathcal F \subset_\e \phi(C(Y))$ (for some sufficiently small $\e$ to be defined).
For each $f \in \mathcal F$, let $g_f \in C(Y)$ be such that $f \approx_\e \phi(g_f)$.
By functional calculus, we may assume that $\mathrm{supp}\, \phi(g_f) \subset U_f$.
Note that $1 \approx_{|\mathcal F|\e} \sum_{f \in \mathcal F} \phi(g_f)$.

Set $Y' := \{y \in Y : \sum_{f \in \mathcal F} g_f(y) \geq 1-2|\mathcal F|\e\}$.
It follows that $\ker \phi = C_0(Z)$ for some open set $Z$ such that $Z \cap Y' = \emptyset$.

Since $Y$ has dimension at most $n$, let $(a_j^{(i)})$ be an $(n+1)$-colourable family of positive elements, that is subordinate to $\{g_f^{-1}((0,\infty)): f \in \mathcal F\}$, such that
\[ \sum_{i,j} a_j^{(i)}|_{Y'} = 1_{Y'}. \]
Set
\[ b_j^{(i)} := \phi(a_j^{(i)}). \]
This is $(n+1)$-colourable since $(a_j^{(i)})$ is.
It is subordinate to $\mathcal U$, since if the support of $a_j^{(i)}$ is contained in $g_f^{-1}((0,\infty))$ then the support of $b_j^{(i)}$ is contained in $U_f$.
Finally, $1-\sum_{i,j} a_j^{(i)} \in C_0(Y \setminus Y') \subseteq C_0(Z)$, and so $\sum_{i,j} b_j^{(i)} = 1$.

(\ref{CoveringTFAEc}) $\Rightarrow$ (\ref{CoveringTFAEa}):
This is clear, since the universal $C^*$-algebra generated by a commuting, $(n+1)$-colourable partition of unity $(b_j^{(i)})_{i=0,\dots,n-1,j=1,\dots,r}$ is $C(Z)$ where $Z$ is the geometric realization of an $n$-dimensional simplicial complex, see Proposition \ref{prop:SimpComplexCstar}.
\end{proof}

\section{Disk-to-sphere retracts}
\label{sec:3}

In this section, we show that low-dimensional, commutative approximations of $C(D^n)$ inside $C(D^n,A)$ implies condition (\ref{MainThm3}) of Theorem \ref{thm:MainThm}.
Condition (\ref{MainThm3}) can be viewed as the existence of noncommutative retracts $D^n \to S^{n-1}$, when $D^n$ is enriched by the space $A$.
We first show that we can (point-norm) approximately satisfy condition (\ref{MainThm3}), which is the more difficult step, and then we use a semiprojectivity argument (in the category of commutative $C^*$-algebras) to obtain condition (\ref{MainThm3}) exactly.

Given an open cover $\mathcal U$ of a topological space $X$, we define the $n$-nerve of $\mathcal U$ to be the following simplicial complex on the vertex set $\mathcal U$: 
\[ N_n(\mathcal U) := \big\{\, \{U_1,\dots,U_k\} \subseteq \mathcal U \mid k \leq n\text{ and } U_1 \cap \cdots \cap U_k \neq \emptyset\big\}. \]
(That is, it is the $(n-1)$-dimensional skeleton of the nerve of $\mathcal U$.)

\begin{lemma}
\label{lem:SnNerve}
Let $\mathcal U$ be an open cover of $S^{n-1}$.
Then there exists an open cover $\mathcal V$ of $D^n$ and a continuous map $\alpha:|N_n(\mathcal V)| \to S^{n-1}$ such that, for each $V \in \mathcal V$, if $V \cap S^{n-1} \neq \emptyset$ then 
\begin{enumerate}
\item there exists $U \in \mathcal U$ such that $V \cap S^{n-1} \subseteq U$; and
\item $\alpha(\mathrm{Star}(V)) = V \cap S^{n-1}$.
\end{enumerate}
\end{lemma}

\begin{proof}
Let $N$ be a simplicial complex with vertex set $X$ such that $|N| \cong S^{n-1}$, such that upon identifying these spaces, each star in $N$ is contained in some set in $\mathcal U$.

View $D^n$ as $CS^{n-1}$ (as in \eqref{eq:Cone}), so that the boundary $S^{n-1}$ is identified with $S^{n-1} \times \{0\}$.
Let $\pi:S^{n-1} \times [0,1] \to D^{n}$ be the quotient map.
For $x \in X$, set
\[ V_0(x) := \pi(\mathrm{Star}(x) \times [0,1/2)), \quad V_1(x) := \pi(\mathrm{Star}(x) \times (0,1)). \]
Then define
\[ \mathcal V := \{V_i(x) \mid i=0,1, x\in X\} \cup \{\pi(S^{n-1} \times (1/2,1])\}, \]
an open cover of $D^n$.

Note that $V_0(x) \cap \pi(S^{n-1} \times (1/2,1]) = \emptyset$, so that if $\{W_1,\dots,W_k\} \in N_{n}(\mathcal V)$ then either:
\begin{enumerate}[(a)]
\item Each $W_j$ is of the form $V_i(x)$ for some $i=0,1$ and $x\in X$; or
\item Some $W_j$ is equal to $\pi(S^{n-1} \times (1/2,1])$ and for every other $j$, $W_j$ is of the form $V_1(x)$ for some $x \in X$.
\end{enumerate}
Let $N'$ be the subcomplex of $N_n(\mathcal V)$ consisting of simplices $\{W_1,\dots,W_k\}$ of the first type; then the map $V_i(x) \to x$ induces a simplicial map $N' \to N$, and thereby a continuous map
\[ |N'| \to |N|=S^{n-1}. \]
We set $\alpha|_{|N'|}$ to be this map.
Then, since $|N_{n}(\mathcal V)|$ is $(n-1)$-dimensional, we may extend this map to all of $|N_{n}(\mathcal V)|$.

For $V \in \mathcal V$, if $V\cap S^{n-1} \neq \emptyset$ then $V= V_0(W)$ for some $W \in \mathcal W$.
In this case, $\mathrm{Star}(V)$ is contained entirely in $|N'|$, so that its image is precisely $W=V \cap S^{n-1}$, so that (ii) holds.
By our choice of $\mathcal W$, (i) also holds.
\end{proof}

The following proposition, making use of the covering dimension result in the previous section, establishes that Theorem \ref{thm:MainThm} (\ref{MainThm1}) implies an approximate version of Theorem \ref{thm:MainThm} (\ref{MainThm3}).

\begin{prop}
\label{prop:ApproxRetract}
Let $A$ be a unital $C^*$-algebra and let $n \in \N$.
If $C(D^n)$ is approximated in $C(D^n,A)$ by commutative $C^*$-algebras whose spectrum has dimension at most $(n-1)$, then the following holds.
For any finite set $\mathcal F \subset C(S^{n-1})$ and any $\e > 0$, there exists a $*$-homomorphism $\phi:C(S^{n-1}) \to C(D^n,A)$ such that
\[ a \approx_\e \phi(a)|_{S^{n-1}} \]
for all $a\in\mathcal F$;
\end{prop}

\begin{proof}
Let $\mathcal F \subset C(S^{n-1})$ and $\e > 0$ be given.
Let $\mathcal U_0$ be an open cover of $S^{n-1}$ such that for every $a \in \mathcal F$, $U \in \mathcal U_0$ and $x,y \in U$, $\|a(x)-a(y)\| < \e$.
There exists an open cover $\mathcal U$ of $S^{n-1}$ such that, for every $k$ and every $V_1,\dots,V_k \in \mathcal U$, if $V_1 \cap \cdots \cap V_k \neq \emptyset$ then there exists $U \in \mathcal U_0$ such that
\[ V_1 \cup \cdots \cup V_k \subseteq U. \]
(This can be done by taking a barycentric refinement of $\mathcal U_0$, see eg.\ \cite[VIII.3]{Dugundji}; alternatively, it is not hard to do by putting a metric on $S^{n-1}$.)

Let $\mathcal V$ be an open cover of $D^n$ and $\alpha:|N_n(\mathcal V)| \to S^{n-1}$ be a continuous map as provided by Lemma \ref{lem:SnNerve}.
By the hypothesis and Proposition \ref{prop:CoveringTFAE}, let $(b^{(i)}_j)_{i=0,\dots,n;\ j=1,\dots,r}$ be an $n$-colourable commuting partition of unity subordinate to $\mathcal V$.
For each $i,j$, let $V^{(i)}_j \in \mathcal V$ be a set containing the support of $b^{(i)}_j$.

As described in Section \ref{sec:SimpComplex}, let $N$ be the nerve of $(b^{(i)}_j)$ (with $\{(i,j)\mid i=0,\dots,n, j=1,\dots,r\}$ as the set of vertices) and let $\psi:C(|N|) \to C^*(\{b^{(i)}_j\})$ be the canonical $*$-homomorphism.
If $\{b^{(i_1)}_{j_1},\dots,b^{(i_k)}_{j_k}\}$ is a simplex in $N$ then evidently $k \leq n$ (by $n$-colourability), and
\[ V^{i_1}_{j_1} \cap \cdots \cap V^{i_k}_{j_k} \neq \emptyset. \]
Thus, $(i,j) \mapsto V^{(i)}_j$ induces a simplicial map from $N$ to $N_{n}(\mathcal V)$, and thereby a continuous map $\beta:|N| \to |N_{n}(\mathcal V)|$.

Define $\phi:C(S^{n-1}) \to C(D^{n},A)$ to be the following composition
\[ C(S^{n-1}) \labelledrightarrow{f \mapsto f \circ \alpha} C(|N_{n}(\mathcal V)|) \labelledrightarrow{f \mapsto f \circ \beta} C(|N|) \labelledrightarrow{\psi} C^*(\{b^{(i)}_j\}) \subseteq C(D^{n},A). \]
Let us now show that $\|a-\phi(a)|_{S^{n-1}}\| \leq \e$ for $a \in \mathcal F$, by showing that $\|a(x)-\phi(a)(x)\| \leq \e$ for every $x \in S^{n-1}$.

Therefore, fix $x \in S^{n-1}$ and $a \in \mathcal F$.
Let the kernel of $\ev_x \circ \phi$ be $C_0(S^{n-1} \setminus Y)$, where $Y \subseteq S^{n-1}$ is closed; thus, $\ev_x \circ \phi$ can be viewed as a representation of $C(Y)$.
Using the fact that $a(x) \in \C$ and $\phi$ is unital, we see that $\|a(x)-\phi(a)(x)\| \leq \sup_{y \in Y} \|a(x)-a(y)\|$.
We shall show that $Y$ is a subset of some open set in $\mathcal U'$, so that we may conclude that $\|a(x)-\phi(a)(x)\| \leq \e$.

Let the kernel of $\ev_x \circ \psi$ be $C_0(|N| \setminus Z)$, so that $Y \subseteq \alpha(\beta(Z))$ (by the definition of $\phi$).
Set $T:= \{(i,j) \mid b^{(i)}_j(x) \neq 0\}$.
By the definition of $\psi$, we can see that $Z$ is contained in the union of stars about vertices $(i,j)\in T$.
The definition of $\beta$ then ensures that $\beta(Z)$ is contained in the union of stars about vertices $V^{(i)}_j$, where $(i,j) \in T$.
By Lemma \ref{lem:SnNerve} (ii), $\alpha(\mathrm{Star}(V^{(i)}_j)) = V^{(i)}_j \cap S^{n-1}$ in $S^{n-1}$, so that
\begin{equation}
\label{eq:ApproxRetractYsub}
 Y \subseteq \alpha(\beta(Z)) \subseteq \bigcup_{(i,j) \in T} V^{(i)}_j \cap S^{n-1}.
\end{equation}

By definition of $T$, and since the support of $b^{(i)}_j$ is contained in $V^{(i)}_j$, we have
\[ x \in \bigcap_{(i,j) \in T} V^{(i)}_j \cap S^{n-1}. \]
By Lemma \ref{lem:SnNerve} (i), $V^{(i)}_j \cap S^{n-1}$ is a subset of some $U^{(i)}_j \in \mathcal U$.
By our choice of $\mathcal U$, this implies that there exists $U \in \mathcal U_0$ such that
\[ \bigcup_{(i,j) \in T} V^{(i)}_j \cap S^{n-1} \subseteq U. \]
Combined with \eqref{eq:ApproxRetractYsub}, we find that $Y \subseteq U$, as required.
\end{proof}

Our next task is to turn the approximate version of Theorem \ref{thm:MainThm} (\ref{MainThm3}) into an exact version, thus completing the proof of Theorem \ref{thm:MainThm} (\ref{MainThm1}) $\Rightarrow$ (\ref{MainThm3}).
This will use a fact about ANRs, namely that nearby maps into an ANR are homotopic \cite[Theorem IV.1.1]{Hu:RetractBook}, applied to the ANR $S^{n-1}$.
We state (and use) this result in the language of commutative $C^*$-algebras.

\begin{lemma} \cite[Theorem IV.1.1]{Hu:RetractBook}
\label{lem:SnHomotopy}
Let $n \in \N$.
There is a finite set $\mathcal F \subset C(S^{n-1})$ and $\e > 0$ such that the following holds:
If $A$ is a commutative $C^*$-algebra and $\phi_0,\phi_1:C(S^{n-1}) \to A$ are $*$-homomorphisms which satisfy
\[ \phi_0(a) \approx_\e \phi_1(a) \]
for all $a\in \mathcal F$, then $\phi_0$ and $\phi_1$ are homotopic, i.e.\ there is a $*$-homomorphism $\phi:C(S^{n-1}) \to C([0,1],A)$ such that
\[ \phi_i = \ev_i \circ \phi \]
for $i=0,1$.
\end{lemma}

\begin{proof}[Proof of Theorem \ref{thm:MainThm} (\ref{MainThm1}) $\Rightarrow$ (\ref{MainThm3})]
Suppose that (\ref{MainThm1}) holds, and thereby so does the conclusion of Proposition \ref{prop:ApproxRetract}.
Let $\mathcal F \subset C(S^{n-1})$ and $\e>0$ be given by Lemma \ref{lem:SnHomotopy}.
Use the conclusion of Proposition \ref{prop:ApproxRetract} to get a $*$-homomorphism
\[ \phi_0:C(S^{n-1}) \to C(D^{n},A) \]
such that $\phi_0(a)|_{S^{n-1}} \approx_\e a$ for all $a\in \mathcal F$.
Note that 
\[ B:=C^*(C(S^{n-1}) \cup \phi_0(C(S^{n-1}))|_{S^{n-1}}) \]
is a commutative subalgebra of $C(S^{n-1},A)$.
Therefore by Lemma \ref{lem:SnHomotopy}, there exists $\psi:C(S^{n-1}) \to C([0,1],B)$ such that
\begin{equation}
\label{eq:DimRetractHomotope}
 \ev_0 \circ \psi = \phi_0|_{S^{n-1}}
\end{equation}
and $\ev_1 \circ \psi$ is equal to the trivial inclusion $C(S^{n-1}) \subseteq B$.

$D^{n}$ is homeomorphic to $(D^{n} \cup [0,1] \times S^{n-1})/\sim$, where $\sim$ identifies each point of $\partial D^{n} \cong S^{n-1}$ with the corresponding point of $\{0\} \times S^{n-1}$.
This provides an identification of $C(D^n,A)$ with a subalgebra $C$ of
\[ C(D^n,A) \oplus C([0,1] \times S^{n-1},A), \]
and we see that \eqref{eq:DimRetractHomotope} ensures that the image of the map $\phi_0 \oplus \psi:C(S^{n-1}) \to C(D^{n},A) \oplus C([0,1] \times S^{n-1},A)$ is contained in $C$.
Thus, we may view $\phi_0 \oplus \psi$ as a map $\phi:C(S^{n-1}) \to C(D^n,A)$.
Then for $a \in C(S^{n-1})$,
\[ \phi(a)|_{S^{n-1}} = \psi(a)|_{\{1\} \times S^{n-1}} = a, \]
as required.
\end{proof}

\section{Kirchberg and R\o rdam's argument}
\label{sec:4}

\begin{proof}[Proof of Theorem \ref{thm:MainThm} (\ref{MainThm3}) $\Rightarrow$ (\ref{MainThm1}).]
This is essentially contained in the proof of \cite[Proposition 3.5]{KirchbergRordam:pi3}.

Assume that (\ref{MainThm3}) holds.
Let us first assume that $X$ is a CW complex.
Given a finite subset $\mathcal F$ of $C(X)$ and $\e>0$, let us take a CW decomposition of $X$ so that each $a\in \mathcal F$ varies by at most $\e$ on each cell.
We may view this decomposition as a canonical surjection
\[ \alpha:\coprod_{k=1}^r D_k \to X, \]
where each $D_k$ is homeomorphic to a disc of dimension at most $n$,
and the restriction of $\alpha$ to $\bigcup D_k^{\circ}$ is one-to-one.
Composition with $\alpha$ provides an injective $*$-homomorphism $C(X) \to C(\coprod_{k=1}^r D_k)$; we will identify $C(X)$ with its image under this map.
For each $k$, define $Y_k \subseteq D_k$ and $\phi_k:C(Y_k) \to C(D_k,A)$ as follows:
If $D_k$ has dimension $n$, set $Y_k:= \partial D_k$ and let $\phi_k$ be as given by (\ref{MainThm3}).
Otherwise, set $Y_k := D_k$ and let $\phi_k:C(D_k) \to C(D_k,A)$ be the inclusion.

Note that in both cases, we have $\partial D_k \subseteq \partial Y_k$ and $\phi_k(a)|_{\partial D_k} = a|_{\partial D_k}$ for all $a \in C(D_k)$.
Set $\hat\phi = \bigoplus \phi_k:C(\coprod_{k=1}^r Y_k) \to C(\bigcup_{k=1}^r D_k)$, and set $Y = \alpha(\coprod Y_k)$.
Identify $C(Y)$ with a subalgebra of $C(\coprod_k Y_k)$ via $\alpha$.
Since $\hat\phi(a)|_{\bigcup \partial D_k} = a|_{\bigcup \partial D_k}$, we see that for $a \in C(Y)$,
\[ \hat\phi(a) \in C(X,A). \]
Hence, $\phi:= \hat\phi|_{C(Y)}$ is a map from $C(Y)$ to $C(X,A)$.
For $x \in X$, the kernel of $\ev_x \circ \phi$ is $C_0(Y \setminus K)$ where
\[ K \subseteq \bigcup \{\alpha(D_k) \mid x \in \alpha(D_k)\}. \]
Since each $a \in \mathcal F$ varies by at most $\e$ on each $\alpha(D_k)$, we see that
\[ \phi(a|_Y)(x) \approx_\e a(x). \]
In particular, $\mathcal F$ is approximated by the image of $C(Y)$, thus concluding the proof in the case that $X=D^n$.

Now, any other $n$-dimensional compact Hausdorff space $X$ is an inverse limit of CW complexes of dimension at most $n$, so that $\mathcal F$ can be approximated (up to $\e/2$), inside $C(X)$, by an algebra $C$ isomorphic to $C(Z)$ for some $n$-dimensional CW complex $Z$.
Hence, we may approximate $\mathcal F$ (up to $\e$) inside $C \otimes A \subseteq C(X) \otimes A$ by a commutative algebra of topological dimension at most $n-1$.
\end{proof}

\begin{proof}[Proof of the last sentence of Theorem \ref{thm:MainThm}]
The above argument, and induction, shows that if Theorem \ref{thm:MainThm} (\ref{MainThm3}) holds, then for any $m \geq n$, the inclusion $C(D^m) \to C(D^m,A^{\otimes m})$ can be approximately factorized as
\[ C(D^m) \labelledrightarrow{\psi} C(\Gamma) \labelledrightarrow{\phi} C(D^m,A^{\otimes (m-n+1)}), \]
where $\Gamma$ has dimension at most $n-1$ and $\psi\phi(f)|_{S^{m-1}} = f|_{S^{m-1}}$.
We may identify $A^{m-n+1}$ with $A$, so that for any $m \in \N$, the inclusion $C(D^m) \to C(D^m,A^{\otimes m})$ can be approximately factorized as
\[ C(D^m) \labelledrightarrow{\psi} C(\Gamma) \labelledrightarrow{\phi} C(D^m,A^{\otimes (m-n+1)}), \]
where $\Gamma$ has dimension at most $n-1$ and $\psi\phi(f)|_{S^{m-1}} = f|_{S^{m-1}}$ (this factorization is trivial if $m \leq n-1$).
Then, a patching-together argument as in the above proof shows that (\ref{MainThm1}) holds with the words ``$n$-dimensional'' removed, when $X$ is a CW complex.
Approximating arbitrary $X$ by CW complexes, again as in the above argument, yields (\ref{MainThm1}) with the words ``$n$-dimensional'' removed, in general.
\end{proof}

\newcommand{\cstar}{$\mathrm C^*$}

\end{document}